\documentclass[10pt]{amsart}
\usepackage{graphics}
\usepackage{amsmath,amsthm}
\usepackage{amsfonts}
\usepackage[a4paper]{geometry}

\newtheorem{theorem}{Theorem}[section]
\newtheorem{proposition}{Proposition}[section]
\newtheorem{lemma}{Lemma}[section]
\newtheorem{corollary}{Corollary}[section]
\newtheorem{definition}{Definition}[section]
\newtheorem{remark}{Remark}[section]

\newcommand{\norm}[1]{\left\Vert#1\right\Vert}

\newcommand{\Real}{\mathbb R}
\newcommand{\ep}{\epsilon}

\newcommand{\be}{\beta}
\newcommand{\lam}{\lambda}

\newcommand{\ga}{\gamma}

\newcommand{\Tor}{\mathbb{T}^2}

\newcommand{\esp}{\vspace{0.5cm}}

\begin{document}
\title{Non-uniformly hyperbolic diffeomorphisms derived from the standard map}
\author{Pierre Berger}
\address{Universit\'e Paris 13, Sorbonne Paris Cit\'e, LAGA, CNRS, (UMR 7539), Universit\'e Paris 8, F-93430, Villetaneuse, France.}
\author{Pablo D. Carrasco}
\address{I.M.P.A., Estrada Dona Castorina 110 CEP 22360 420. Rio de Janeiro. Brazil.}
\thanks{The first author is partially supported by the Balzan Research Project of J. Palis. The second author is supported by a C.N.P.Q. grant.}
\date{}

\begin{abstract}
We prove that the system resulting of coupling the standard map with a fast hyperbolic system is robustly non-uniformly hyperbolic.
\end{abstract}
\maketitle
\section{Introduction and Statement of the Results}\label{Sect.1}

The theory of uniform hyperbolicity, or ``Axiom A'' diffeomorphisms is presently almost complete. It encompasses various examples \cite{SmaleBull} and implies strong ergodic properties \cite{EqStaAnosov}.
However, this theory does not include many chaotic phenomena, especially due to obstructions which are topological  (existence of the stable bundle on the basin of attractors) or  geometrical (robust tangencies).

To include more examples, Y. Pesin has introduced the very general concept of non-uniform hyperbolicity. A diffeomorphism $f$ of a manifold $M$ is \emph{non-uniformly hyperbolic} (NUH) with respect to an invariant ergodic probability $\mu$ if  $\mu$-almost every $x \in M$ satisfies that for every  unit vector $v\in T_xM$:
\begin{equation}\label{lyapnonzero}\lim_{n\to \infty} \frac1n \log\|df^n(v)\|\not=0.\end{equation}
In fact, by Oseledec's Theorem, the above limit exists for $\mu\text{ a.e. }x$; $T_xM$ is the direct sum of $df$-invariant subspaces, and on each of these subspaces the limit converges to a single value called the \emph{Lyapunov exponent}. Hence $f$ is non-uniformly hyperbolic with respect to $\mu$ if and only if its Lyapunov exponents are non-zero.

A case of particular interest is when the NUH probability $\mu$ is \emph{physical}, that is, when the basin of $\mu$ (formed by the points $x$ for which the Birkhoff sum $\frac1n \sum_k \delta_{f^k(x)}$ converges weakly to $\mu$) has positive Lebesgue measure. We say then that $f$ has a NUH \emph{physical measure}. Such measures have several properties \cite{LyaPesin}, although their ergodic properties depend on additional hypotheses in general.

Let us first give a list of our favorite examples of NUH physical measures:
\begin{itemize}
\item[$(i)$] Jakobson's Theorem \cite{Jakob}: for a Lebesgue positive set of parameters $c\in \mathbb R$, the quadratic family $x^2+c$ has a NUH physical measure.
\item[$(ii)$] Rees' Theorem \cite{ReesThm}: for a Lebesgue positive set of rational functions (of degree $d\geq 2$), there exists a NUH physical measure.
\item[$(iii)$] The Benedicks-Carleson Theorem \cite{dynHenon},\cite{SRBHenon},\cite{Abundance}: for a Lebesgue positive set of parameters $(a,b)$ with $b$ small, the H\'enon map $$h_{ab}(x,y)=(x^2+y+a,-bx)$$ has a NUH physical measure.
\end{itemize}
All these examples are \emph{abundant}: they hold for a Lebesgue positive set of parameters (of generic families). Moreover, the (real or complex) dimension of the support of the measure is one or close to one.

An interesting candidate to be on this list is the Chirikov-Taylor Standard map family $(\mathbf{s}_{r})_{r\in\Real}$, which is formed by conservative maps of the  $2$-torus $\Tor=\mathbb R^2/2\pi \mathbb Z^2$:
$$\mathbf{s}_{r}(x,y)=(2x-y+ r\sin( x),x ).$$
It is conjectured that for a Lebesgue positive set of large parameters $r$, the map $\mathbf{s}_{r}$ is NUH with respect to the Lebesgue measure (and in particular, it has a NUH physical measure). As a consequence, the Hausdorff dimension of the corresponding measure's support would be two.



This conjecture is very hard: the dynamics of the standard map is not understood for any parameter $r\not=0$. On the other hand, the important negative result of P. Duarte \cite{Plenty} states that for a residual set of parameters in $r\geq r_0$, infinitely many elliptic islands coexist (see fig. \ref{KAMisland}). J. De Simoi  \cite{DeS11} showed that this generic set of parameters has Hausdorff dimension at least $1/4$. A. Gorodetski gave a description of a ``stochastic sea'' for a generic set of these parameters \cite{StochasticSea}: his impressive construction showed that for such parameters there exists an increasing sequence of hyperbolic sets with Hausdorff dimension converging to $2$, and such that elliptic islands accumulate on them.

\begin{figure}
	\centering
		\includegraphics{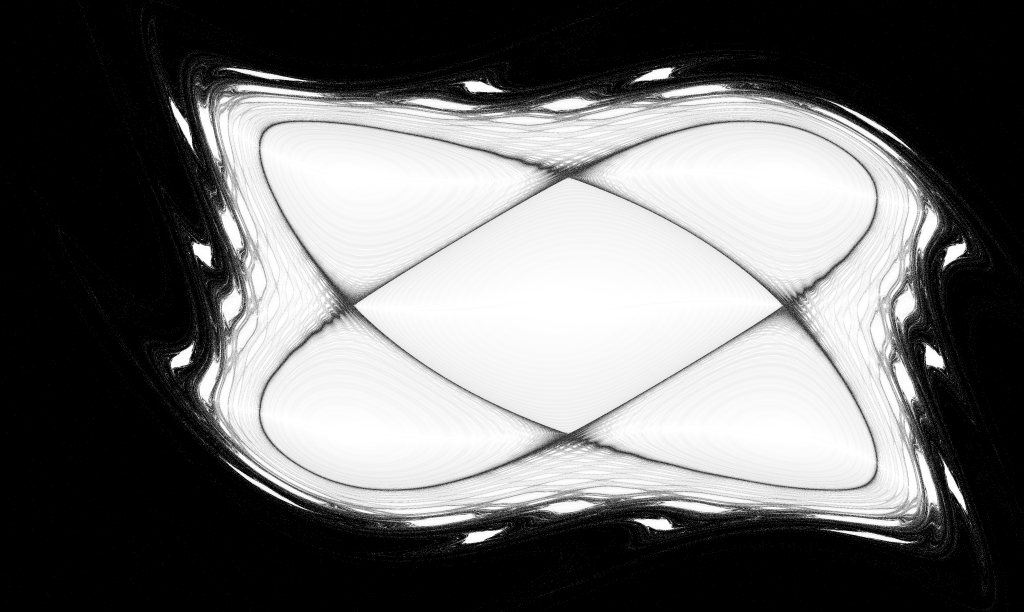}
	\caption{Lyapunov exponent for the parameter $r=-0.364$ of the standard map,
	white regions are expected to be elliptic islands whereas the black regions are possibly the homoclinic web of a NUH attractor.}
	\label{KAMisland}
\end{figure}

Another way to construct examples of NUH physical measures is to work with dynamics which (locally) fiber over a uniformly hyperbolic one. Such techniques have been used notably by \cite{RobTranT4}, \cite{Multinonhyp}, and \cite{Patho} to produce new examples which are robustly non-uniformly hyperbolic.

\begin{definition} A map $f$ of a compact manifold $M$ is  $\mathcal{C}^s$-\emph{robustly non-uniformly hyperbolic} if there exists a $\mathcal{C}^s$-neighborhood $U$ of $f$ such that every map $g\in U$ has a NUH physical measure.
\end{definition}

Recently a new example has been given  \cite{ExtLyaExp}, namely a non-hyperbolic ergodic toral automorphism for which most symplectic perturbations are NUH.  The techniques developed there are also suitable for dealing with some conservative cases, pushing forward the method developed in \cite{Patho} for volume preserving diffeomorphisms.

The above list of examples is almost exhaustive. There are very few known NUH physical measures
which  have both  non-uniformly contracting and expanding directions (and which are abundant). In particular, the sophisticated techniques used by Benedicks and Carleson to study the H\'enon attractor are not completely satisfactory for a general theory since they need a small determinant $b$. In order to complete the theory of non-uniform hyperbolicity, more examples are needed.

In this work we show the existence of a non-hyperbolic robust conservative NUH diffeomorphism, adding a
different type of example to the above list.



Let $A\in SL_2(Z)$\ be a hyperbolic matrix with eigenvalues $\lambda<1<1/\lambda$. Consider the manifold $M=\Tor\times \Tor$\ with coordinates $m=(x,y,z,w)$, and the analytic diffeomorphism $f_N:M\rightarrow M$\ given by

$$
f_N(m)=(\mathbf{s}_{N}(x,y)+P_x\circ A^{N}(z,w),A^{2N}(z,w))
$$
where $N\ge 0$, and $P_x$\ is the projection of $\Real^2$ to the $x$-axis $\Real\cdot (1,0)$.

\begin{theorem}\label{main}
There exist $N_0$ and $c>0$\ such that for every $N\geq N_0$,\ the map $f_N$ satisfies for Lebesgue a.e. $m\in \Tor\times \Tor$ and every unit vector $v\in \mathbb R^4$:
\[\lim_{n\to \infty} \big|\frac1n \log\|Tf_N^n(v)\|\big|>c\log N.\]

Moreover the same holds for every conservative diffeomorphism in a $\mathcal{C}^{2}$-neighborhood of $f_N$.
\end{theorem}
A direct consequence of Pesin's ergodic decomposition Theorem for NUH measures \cite{LyaPesin} is:
\begin{corollary}\label{coro}
There exists $N_0$\ such that for every $N\geq N_0$,
in a neighborhood $U$ of $f$ in the space of
$\mathcal{C}^{2}$-conservative maps, every map has a NUH physical measure.
\end{corollary}
A similar result was announced in \cite{RobNon}. We hope that this new example will
be helpful to study the standard map, notably by wondering about the ergodicity and the mixing property of the Lebesgue measure with the dynamics $f_N$ (for most of the volume preserving perturbations of $f_N$ the Lebesgue measure is ergodic, as explained in the last section).

At the end of the paper we illustrate how the techniques of the paper can be used to show the same results for families of maps $(g_{N_1,N_2})_{N_1\le N_2}$ of the following form: for $p\in \mathbb Z$ and $L\in SL_2(\mathbb Z)$ such that $L_{[1,1]}\not=0$, consider
\[g_{N_1,N_2}\colon (x,y,z,w)\in  M \mapsto (px-y+N_1 \sin(x), x,0,0) +(L(z,w) , A^{N_2}(z,w)).\]


%


The standard family $(\mathbf{s}_{r})_r$ is related to numerous physical problems; see for instance \cite{UnivSta}, \cite{I80}, \cite{SS95}.

On the other hand, there is a physical interpretation of the coupling of a dynamical system $\phi$ with an Anosov diffeomorphism; the dynamics of the Anosov map would model the microscopic forces which perturb the macroscopic scale endowed with the dynamics $\phi$. See \cite{MacKay12}.

%
%
%
%
%
%

\section{Preliminaries}\label{Sect.2}

Along the proof $N$ will be supposed to be larger and larger in order to satisfy finitely many  conditions.

In this section we collect some basic facts about the dynamical properties of $f_N$\ and its inverse. Let us begin by noting the following symmetry property.

\begin{lemma}\label{fmenosuno}
The map $f_N^{-1}$\ is conjugated to the map
$$(x,y,z,w)\mapsto (\mathbf{s}_N(x,y)+ P_x\circ A^{-N}(z,w), A^{-2N}(z,w)).$$
\end{lemma}

\esp

\begin{proof}

Consider the involutions

$$R(x,y)=(y,x)\quad \quad
J(x,y)=(x,2x-y+{N} \sin{ x})$$
and observe that  $R\circ J=\mathbf{s}_N$, hence $\mathbf{s}_N^{-1}=J\circ R$\ (this is the \emph{reversibility} of the standard map). To seek the inverse of $f_N$, we put

$$(a,b,c,d)= (\mathbf{s}_N(x,y)+ P_x\circ A^{N}(z,w), A^{2N}(z,w)).$$

Thus $(z,w)=A^{-2N}(c,d)$, and if we denote by $P_y$\ the projection to the $y$-axis, we get
\begin{multline*}
(x,y)=\mathbf{s}_N^{-1}((a,b)-P_x\circ A^{N}(z,w))=\mathbf{s}_N^{-1}((a,b)-P_x\circ A^{-N}(c,d))\\
=J(R(a,b)-P_y\circ R\circ  A^{-N}(c,d))=J\circ R(a,b)+P_y\circ R\circ  A^{-N}(c,d)
\end{multline*}
since $J((i,j)-(0,k))=J(i,j)+(0,k)$.  We conclude
\begin{equation}
f^{-1}_N(a,b,c,d)=(J\circ R(a,b)+P_y\circ R\circ  A^{-N}(c,d), A^{-2N}(c,d)).
\end{equation}

Finally, if $\hat R$ is the involution  $\hat R(a,b,c,d)=(R(a,b),(c,d))$, we obtain

\begin{equation}
\hat R\circ f^{-1}\circ \hat R(a,b,c,d)=(\mathbf{s}_N(a,b)+ P_x\circ A^{-N}(c,d), A^{-2N}(c,d)).
\end{equation}

\end{proof}

This Lemma allows us to establish simultaneously properties for $f_N$\ and $f_N^{-1}$. For example, to prove that $f_N$\ is NUH it suffices to show that it has two positive exponents.

As mentioned, the map $f_N$\ will be proved to be partially hyperbolic. We recall here this notion and the important stable manifold theorem.

\begin{definition}
Let  $M$\ be a closed Riemannian manifold. A $\mathcal{C}^1$\ diffeomorphism  $f: M \rightarrow  M$\ is \emph{partially hyperbolic} if there exists a continuous splitting of the tangent bundle of the form  $$TM=E^u\oplus E^c\oplus E^s$$\ where both bundles $E^s,E^u$\ have positive dimension, and such that
\begin{enumerate}
  \item All bundles $E^u,E^s,E^c$\ are $df$-invariant.
  \item For every $m\in M$, for every unitary vector $v^{\sigma}\in E^{\sigma}_m, \sigma=s,c,u$,
  \begin{gather*}
  \norm{d_mf(v^{s})}<1<\norm{d_mf(v^{u})}\\
  \norm{d_mf(v^{s})}< \norm{d_mf(v^{c})}<\norm{d_mf(v^{u})}
  \end{gather*}
\end{enumerate}
\end{definition}

The bundles $E^s,E^u,E^c$\ are called the \emph{stable, unstable} and \emph{center} bundle respectively. The reference \cite{PesinLect} contains a good account of partial hyperbolicity.
\begin{theorem}[Stable Manifold Theorem]
If $f\in Diff^r(M)$\ is partially hyperbolic then the bundles $E^s,E^u$\ are (uniquely) integrable to continuous foliations $\mathcal{F}^s=\{W^s(m)\},\mathcal{F}^u=\{W^u(m)\}$\ resp. called the \emph{stable} and the \emph{unstable} foliations, whose leaves are $\mathcal{C}^r$\ immersed submanifolds.
\end{theorem}

See Section 4 in \cite{PesinLect} for the proof of this theorem. Let us come back to the study of $f_N$.

Consider a normal eigenbasis $\{e^s,e^u\}$\ of $A$\ corresponding to the respective eigenvalues $\lam<1<\frac{1}{\lam}=:\mu$. Put $E^u_A=span\{e^u\}$ and $E^s_A=span\{e^s\}$. We also consider the following vector subspaces of $\Real^4$:
\[E^u_0=\{0\}\times E^u_{A},\quad E^s_0=\{0\}\times E^s_{A},\quad E^c=\Real^2\times \{0\}.\]
For each $m\in M, p\in \Tor$ we identify $T_mM \equiv\Real^4,\ T_p\Tor\equiv\Real^2$.

\begin{proposition}\label{partialhyp}
If $N$\ is sufficiently large the map $f_N$\ is  partially hyperbolic with center bundle $E^c$, and one dimensional stable and unstable direction $E^s$ and $E^u$ respectively. Moreover,
there exists a differentiable vector field $m\in M\mapsto (\alpha_m, e^u)\in \mathbb R^2\times \mathbb R^2$ such that:
\begin{itemize}
\item for every $m\in M$, $df_N(\alpha_m,e^u)= \mu^{2N}(\alpha_{f(m)},e^u)$, and so  $(\alpha_m, e^u) $ is tangent to $E^u$.
\item $\|\alpha_m-\lambda^N P_x(e^u)\|\le \lambda^{2N}$.
\end{itemize}
%
\end{proposition}

\esp
A consequence of Proposition \ref{partialhyp} is:

\begin{corollary}\label{partialcoro}
If $N$ is sufficiently large then for all $m\in M, v=(v_x,v_y,v_z,v_w)\in E^{u}_m\setminus\{0\}$ it holds
$$
\lam^{N}(\|P_x(e^u)\|-3\lam^N)
\leq \frac{|v_x|}{\norm{v}}\leq
\lam^{N}(\|P_x(e^u)\|+3\lam^N)
$$
with $\norm{P_x(e^{u})}>0$.
\end{corollary}

Similar statements hold for the stable bundle $E^s_N$.

\begin{proof}[Proof of Proposition \ref{partialhyp}]

The derivative of $f_N$\ at a point $m=(x,y,z,w)\in M$\ is given by
$$
d_mf_N=\begin{pmatrix}
           d_{(x,y)}\mathbf{s}_N &  P_x\circ A^{N} \\
           0 & A^{2N} \\
         \end{pmatrix}
$$
From here we deduce that $E^c$\ is a $df$-invariant bundle, and furthermore the action on a vector $v=(v_x,v_y,0,0)\in E^c_m$\ is given by

\begin{equation}\label{accentral}
d_mf_N(v_x,v_y,0,0) = (d_{(x,y)}\mathbf{s}_N(v_x,v_y),0,0)
\end{equation}

In particular (by reversibility of the standard map), we have

\begin{equation}\label{normadfcentro}
\frac{1}{2N}\norm{v}\leq \norm{d_mf_N(v)} \leq 2N\norm{v}
\end{equation}

Let us study the unstable direction. For $m\in M$, consider:
\[\alpha_m:= \lambda^N \sum_0^{\infty} \lambda^{2Nk} d_{(x_{-1},y_{-1})}\mathbf{s}_N \cdots  d_{(x_{-k},y_{-k})}\mathbf{s}_N(P_x(e^u)),\]
with $(x_j,y_j,z_j,w_j):=f_N^j(m)$.

We remark that:
\[E^u(m):= \mathbb R \cdot ( \alpha_m , e^u),\]
is a differentiable line field on $M$, since $\|\lambda^{2N} d\mathbf{s}_N\|< \lambda^N/2 <1$ (and similarly for the derivatives). Moreover:
\[ \|\alpha_m-\lambda^N P_x(e^u)\|\le \lambda^{2N}\]
	and
\[d_m f_N (\alpha_m, e^u)=(d_{(x_0,y_0)} \mathbf{s}_N\alpha_m + \mu ^N P_x(e^u), \mu^{2N} e^u)=\mu^{2N}(\alpha_{f_N (m)}, e^u).\]
Indeed $\frac{d_{(x_0,y_0)} \mathbf{s}_N\alpha_m + \mu ^N P_x(e^u)}{\mu^{2N}}=\lambda^{2N}\left(d_{(x_0,y_0)} \mathbf{s}_N\alpha_m + \mu ^N P_x(e^u)\right)$ is equal to
\begin{equation*}
\begin{split}
&\lambda^N P_x(e^u) + \lambda^{2N} d_{(x_0,y_0)}\mathbf{s}_N  (\lambda^N \sum_0^{\infty} \lambda^{2Nk}d_{(x_{-1},y_{-1})}\mathbf{s}_N \cdots  d_{(x_{-k},y_{-k})}\mathbf{s}_N(P_x(e^u))\\
&=\lambda^N P_x(e^u) + \lambda^N \sum_0^{\infty} \lambda^{2N(k+1)} d_{(x_{0},y_{0})}\mathbf{s}_N \cdots d_{(x_{-k},y_{-k})}\mathbf{s}_N(P_x(e^u))\\
&=\alpha_{f_N(m)}.
\end{split}
\end{equation*}

\end{proof}

\begin{proof}[Proof of Corollary \ref{partialcoro}.]
If $v=(\alpha_m , e^u)$, then Proposition \ref{partialhyp} gives
\[\lam^{N}(\|P_x(e^u)\|-\lam^N)\leq {|v_x|}\leq\lam^{N}(\|P_x(e^u)\|+\lam^N)\]
\[1-\lam^{2N}-\lam^N \|P_x(e^u)\|\leq \norm{v}\leq 1+\lam^{2N}+\lam^N \|P_x(e^u)\|\]
which imply both inequalities. Note that $P_x(e^u)\neq 0$ since $A$ is a hyperbolic matrix.
\end{proof}

From now on we work with $N$\ for which the previous Proposition and Corollary hold.

\begin{remark}
Observe that the homological action $(\mathbf{s}_N)_{\ast}=\begin{pmatrix}
2 & -1\\
1 & 0
\end{pmatrix}\colon H_1(\mathbb T^2;\Real)\rightarrow H_1(\mathbb T^2;\Real)$\ of $\mathbf{s}_N$ is unipotent,
and so the one of $f_N$
is not hyperbolic. Hence by \cite{AnosovTori} the map $f_N$\ is not uniformly hyperbolic.
\end{remark}
Let us identify $E^c$ to $\mathbb R^2$. By equation \eqref{accentral}, we see that $d_{m}f_N|E^c:\mathbb{R}^2\rightarrow \mathbb{R}^2$\ coincides with $d_{(x,y)}\mathbf{s}_N$.  In what follows we will identify:
\[df_N|E^c=d\mathbf{s}_N.\]
 Note that $\norm{d^2\mathbf{s}_N}\leq N$, hence using \eqref{normadfcentro} we conclude

\begin{equation}\label{normacentro}
\frac{\norm{u}}{2N}\leq \norm{df_N(u)} \leq 2N\norm{u},\quad \forall u\in E^c,\quad \norm{d^2f_N}\leq N.
\end{equation}

In general the center bundle of a partially hyperbolic system is not integrable (see \cite{Nodyncoh}). However, in our case by construction $E^c$\ integrates to the smooth fibration by tori $(\Tor\times\{w\})_{w\in\Tor}$, hence by combining
Theorems 7.1 and 7.2 in  \cite{HPS} we obtain.
\begin{corollary}[Hirsch-Pugh-Shub]\label{dc}
There exists a neighborhood $U\subset Diff^2(M)$\ of $f_N$\ such that every $g\in U$\ is partially hyperbolic with center bundle $E^c_g$\  integrable to a continuous fibration by $\mathcal C^2$-tori depending continuously on the base point.
\end{corollary}

\section{Central Exponents.}\label{Sect.3}

We will first prove non-uniform hyperbolicity of the map $f_N$, and later in Section $6$\ we will indicate the necessary modifications to our arguments for dealing with perturbations. To make the notation less cumbersome we write $f=f_N$.

In the unstable direction $f$\ has one positive Lyapunov exponent with respect to the Lebesgue measure. We are now seeking another positive exponent in the center direction. Let us introduce a few notions to perform this.

\begin{definition}A $u$-curve for $f$\ is  a curve $\ga=(\ga_x,\ga_y,\ga_y,\ga_w ):[0,2\pi] \rightarrow M$, tangent to $E^u$ and such that
$$\frac{d\ga}{dt}(t)=\frac{(\alpha_{\ga(t)},e^u)}{\lambda^N  \|P_x(e^u)\|}\quad \forall t\in [0,2\pi].$$
\end{definition}

By Corollary \ref{partialcoro}, $\|\frac{d\ga_x}{dt}(t)\|\in [1-N^{-1},1+ N^{-1}]$, and thus the curve $\ga$ makes approximately one turn of $M$ in the $x$ while making many more turns (on the order of $\frac1{\lam^N}$) around
the base torus $\{0\}\times\Tor$.


Every $u$-curve $\gamma$ is a segment of an unstable leaf. Moreover,
\begin{equation}\label{tiempo}
\frac{d (f^k\circ\ga)}{dt}(t) = \mu^{2kN} \frac{(\alpha_{f^k(\ga (t))},e^u)}{\lambda^N  \|P_x(e^u)\|}
, \quad \forall t\in[0,2\pi],\;\forall k\ge 0.
\end{equation}
 Hence, with $[\mu^{2kN}]$  the integer part of $\mu^{2kN}$, the curve $f^k\circ\ga$\ can be written as a concatenation
$$
f^k\circ\ga=\ga_1^k\ast\cdots\ga^k_{[\mu^{2kN}]}\ast \ga^k_{[\mu^{2kN}]+1}
$$
where $\ga_j^k,$ $j\in\{1,\ldots,  [\mu^{2kN}]\}$\ are $u$-curves and $\ga^k_{[\mu^{2kN}]+1}$\ is a segment of a $u$-curve.

To prepare the perturbative case, we look at $1/2$-H\"older vector fields on $u$-curves.

\begin{definition}
An \emph{adapted field} $(\ga,X)$ is the pair of
 a $u$-curve $\ga$ and of a unit vector field  $X$ satisfying:
\begin{enumerate}
\item $X$\ is tangent to the center direction.
\item $X$ is $(C_X,1/2)$-H\"older along $\ga$, meaning
      $$
      \forall m,m'\in \ga,\quad \norm{X_{m}-X_{m'}}\leq C_Xd(m,m')^{1/2}
      $$
      with $C_X<20 N^2\lam^{N}$, and where the distance $d(m,m')$ is measured along $\ga$.
\end{enumerate}
\end{definition}

From Proposition \ref{partialhyp}  one deduces that the length of a $u$-curve is bounded by $$\frac{2\pi (1+\lam^N) }{\lam^N\norm{P_x(e^u)}},$$ hence:
\begin{corollary}\label{variation} If $N$\ is sufficiently large then for every adapted field $(\gamma, X)$, for all $m,m'\in \gamma$:
\[\norm{X_{m}-X_{m'}}<\lam^{N/3}.\]
\end{corollary}
\begin{proof}
Indeed,
$$
\norm{X_{m}-X_{m'}}<20N^2\lam^N\left(\frac{2\pi (1+\lam^N) }{\lam^N\norm{P_x(e^u)}}\right)^{1/2}<\lam^{N/3}.
$$
\end{proof}

Fix an adapted field $(\ga,X)$\  and let $Y^k:=\frac{(f^k)_{\ast}X}{\norm{(f^k)_{\ast}X}}$. Denote by $d\gamma$\ the Lebesgue measure induced on\footnote{That is $d\gamma(t)=\norm{\frac{d\gamma}{dt}(t)} dt.$} $\gamma$\ and by $|\ga|$\ the length of $\ga$. We compute
\begin{equation*}
\begin{split}
I_n^{\ga,X} &:=
\frac1{|\gamma|}\int_{\ga} \log{\norm{d_mf^nX_m}}d\gamma\\
&=\sum_{k=0}^{n-1}\frac1{|\gamma|}\int_{\ga} \log{\norm{d_{f^km}f(Y^k\circ f^k(m))}}d\gamma\\
&=\sum_{k=0}^{n-1}\frac{1+\epsilon_k}{\mu^{2Nk}|\gamma|}\int_{f^k\circ\ga} \log{\norm{d_{m}f(Y^k)}}d(f^k\circ\gamma),\quad \epsilon_k \in [-2\lambda^{2N},2\lambda^{2N}]
\end{split}
\end{equation*}

where in the last equality we have used a change of variables together with the fact that $\norm{\frac{d(f^k \circ \ga)}{dt}(t)}/(\mu^{2Nk}\norm{\frac{d\ga}{dt}(t)})\in [1-2\lambda^{2N},1+2\lambda^{2N}]$ by Proposition \ref{partialhyp} and (\ref{tiempo}).

We conclude
\begin{multline}\label{poursct6}
I_n^{\ga,X}=\sum_{k=0}^{n-1}\frac{1+\epsilon_k}{\mu^{2Nk}|\ga|}\left(\sum_{j=0}^{[\mu^{2kN}]}\int_{\ga_j^k} \log{\norm{d_{m}f(Y^k)}}d\gamma_{j}^k\right.\\
\left.+\int_{\ga^k_{[\mu^{2kN}]+1}} \log{\norm{d_{m}f(Y^k)}}d\ga^k_{[\mu^{2kN}]+1}\right)
\end{multline}

A crucial point is that all pairs $(\ga_j^k,Y^k|\ga_j^k)$\ for $1\leq j\leq [\mu^{2Nk}]$\ are adapted fields.

\begin{lemma}\label{adapted} For $N$\ sufficiently large, for every adapted field $(\ga,X)$, for every $k\ge 1$, for every $1\leq j \leq [\mu^{2N}]$, the pair $(\ga^k_j,Y^k|\ga_j^k)$ is an adapted field.
\end{lemma}

\begin{proof}
The only non trivial assertion  is that $Y$\ is $(C_Y,1/2)$-H\"older with $C_Y<20N^2\lambda^{N}$. Let us do the proof for $k=1$. We write $Y:=Y^1$, and observe that
\begin{equation*}
\forall m,m'\in M,\quad\norm{d_mf(X_m) -d_{m}f(X_{m'})}\leq 2N\norm{X_m-X_{m'}}\le  2N C_X d(m,m')^{1/2}
\end{equation*}
\begin{equation*}
\forall m,m'\in M,\quad\norm{d_mf(X_{m'}) -d_{m'}f(X_{m'})}\leq Nd(m,m')
\end{equation*}
which is less than $4Nd(m,m')^{1/2}$ if $d(m,m')\leq 1$. If $d(m,m')>1$ we remark that:
\begin{equation*}
\norm{d_mf(X_{m'}) -d_{m'}f(X_{m'})}\leq 4N \le 4N d(m,m')^{1/2}.
\end{equation*}
 This gives, by the triangular inequality:
%
%
\begin{equation}\label{holder}
\forall m,m'\in M,\quad\norm{d_mf(X_m) -d_{m'}f(X_{m'})}\leq (4N+2NC_X)d(m,m')^{1/2}.
\end{equation}
For $m,m'\in \ga^1_j$, we compute using the triangular inequality and the previous equation:\
\begin{equation*}
\begin{split}
\norm{Y_m-Y_{m'}}&=\frac{1}{\norm{f_{\ast}X_m}\norm{f_{\ast}X_{m'}}}\Big\|\norm{f_{\ast}X_{m'}}f_{\ast}X_m -\norm{f_{\ast}X_m}f_{\ast}X_{m'}\Big\|\\
&\leq\frac{1}{\norm{f_{\ast}X_m}\norm{f_{\ast}X_{m'}}}\left(\Big\|\norm{f_{\ast}X_{m'}}f_{\ast}X_m-\norm{f_{\ast}X_{m'}}f_{\ast}X_{m'}\Big\|\right.\\
&+\left.\hspace{3.4cm}\Big\|\norm{f_{\ast}X_{m'}}f_{\ast}X_{m'}-\norm{f_{\ast}X_m}f_{\ast}X_{m'}\Big\|\right)\\
&\leq\frac{2}{\norm{f_{\ast}X_m}}\norm{f_{\ast}X_m-f_{\ast}X_{m'}}\\
&=\frac{2}{\norm{f_{\ast}X_m}}\norm{d_{f^{-1}m}f(X_{f^{-1}m})-d_{f^{-1}m'}f(X_{f^{-1}m'})}\\
\end{split}
\end{equation*}
Hence by (\ref{holder}) and Proposition \ref{partialhyp}, it comes
\[\norm{Y_m-Y_{m'}}
\leq \frac{2(4N+2NC_X)}{\norm{f_{\ast}X_m}}d(f^{-1}m,f^{-1}m')^{1/2}
\leq \frac{4N(2+C_X)\lambda^N(1+\lambda^N)}{\norm{f_{\ast}X_m}}d(m,m')^{1/2}.\]
Since $\norm{df|_{E^c}}\geq \frac1{2N}$, we finally get
$$\norm{Y_m-Y_{m'}}\leq 8N^2(2+C_X)\lam^N(1+\lambda^N) d(m,m')^{1/2}<20N^2\lambda^{N} d(m,m')^{1/2}
$$
if $N$\ is sufficiently large. The general case $k>1$ follows by induction.
\end{proof}

Hence we have constructed a dynamics on the subset of adapted fields.

Let $\ga$\ be a u-curve and $X$\ a vector field tangent to the center direction.
The following Proposition is fundamental.
\begin{proposition}\label{ergo}
Suppose that there exists $C>0$ with the following property: for every $u$-curve $\ga$ there exists a vector field $X$  such that $(\gamma, X)$ is an adapted field and such that for all $n\geq 0$ large\
$$
\frac{I_n^{\ga,X}}{n}>C.
$$
Then the map $f$\ has a positive Lyapunov exponent greater than $C/2$ in the center direction at Lebesgue almost every point.
\end{proposition}

\begin{proof}
Lebesgue almost every point of $M$ is regular, i.e. their Lyapunov exponents are well defined.
Let $B$ be the subset of regular points at which the map $f$\ does not have a positive Lyapunov exponent greater than $C/2$ in the center direction.
For the sake of contradiction, assume that $B$\ has positive Lebesgue measure.
 By absolute continuity of the (strong) unstable foliation (see Chapter 7 in \cite{PesinLect}) and Fubini's Theorem, there exists an unstable manifold $\gamma$ which intersects $B$ at a subset of positive $1$-Lebesgue measure. Let $b$ be a density point of this subset of $\gamma$.
 For $\epsilon$\ small let $\ga^{\epsilon}=(\ga_x^{\epsilon},\ga_y^{\epsilon},\ga_y^{\epsilon},\ga_w^{\epsilon} ):[-\epsilon,\epsilon]\rightarrow M$\ be the curve tangent to $E^u_f$ such that $\frac{d\ga}{dt}(t)=\frac{(\alpha_{\ga(t)},e^u)}{\lambda^N  \|P_x(e^u)\|}\ \forall t$\ and $\ga^{\epsilon}(0)=b$, and denote by $Leb$\ the Lebesgue measure on its image.

Then
$$
\frac{Leb(\ga^{\ep}\cap B)}{Leb(\ga^{\ep})}\xrightarrow[\ep\rightarrow 0]{} 1.
$$

Take $\ep=2\pi{\lam^{2Nk}}$\ with $k$ large so that $Leb(\ga^{\ep}\cap B^c)<\frac{C}{4\log{2N}}Leb(\ga^{\ep})$: then the curve $f^k\ga^{\ep}$\ is a $u$-curve (by Prop. \ref{partialhyp}). Consider $X_{\ep}$\ so that $(f^k\ga^{\ep},f^k_{\ast}X_{\ep})$\ satisfies the Proposition hypothesis  and let $\chi(m)=\limsup_{n\rightarrow\infty}\frac{\log{\norm{d_mf^n(X_{\ep})}}}{n}$. Observe that $\chi(m)<\frac{C}{2}$ if $m\in B$. By hypothesis, we obtain
\begin{equation*}
\begin{split}
\int_{\ga^{\ep}}\chi  d\ga^{\ep}&=\int_{f^k\ga^{\ep}}\chi\circ f^{-k} \frac{1}{\norm{\frac{df^k\ga^{\ep}}{dt}}}d(f^k\ga^{\ep})=\frac{1+\epsilon_k}{\mu^{2Nk}}\int_{f^k\ga^{\ep}}\chi\circ f^{-k} d(f^k\ga^{\ep})\\
&\ge \lam^{2Nk}(1+\epsilon_k)\limsup_{n\rightarrow\infty} \int_{f^k\ga^{\ep}} \frac{\log{\norm{d_mf^n(    f^k_{\ast} X_{\ep})}}}{n} d(f^k\ga^{\ep}) \ge C(1+\epsilon_k)|f^k\ga^{\ep}|\lam^{2Nk}\\
&\ge C (1+\epsilon'_k) Leb(\ga^{\ep}),
\end{split}
\end{equation*}
with $|\epsilon_k|,|\epsilon'_k|\le 2\lambda^{2N}$.
On the other hand, by equation \eqref{normacentro}, $\chi\leq \log{2N}$\ and thus
\begin{equation*}
\begin{split}
\int_{\ga^{\ep}}\chi d\ga^{\ep}&=\int_{\ga^{\ep}\cap B}\chi dLeb+\int_{\ga^{\ep}\cap B^c}\chi dLeb \\
&\leq \frac C{2} Leb(\ga^{\ep}) + \log{2N}Leb(\ga^{\ep}\cap B^c)<\frac34 C Leb(\ga^{\ep})
\end{split}
\end{equation*}
which gives the contradiction.
\end{proof}

\begin{remark}
In  the previous Proposition, as in Proposition \ref{ergog}, it is enough to have the inequality for almost every $u$-curve. The proof is the same.
\end{remark}

We are led to study the value of
$$
E(\ga,X):=\frac1{|\gamma|}\int_{\ga} \log \norm{d_mf(X)}d\ga
$$
for  adapted fields $(\ga,X)$.

\section{The basis in the Central Direction}\label{Sect.4}
The study of $E(\ga,X)$\ will be achieved by introducing a convenient basis of $E^c$. For $m=(x,y,z,w)\in M$\ define
\[\Omega(m)=2+N\cos  x.\]
One verifies directly that $d_mf|E^c$\ is represented by the matrix
$$
d_mf|E^c=\begin{pmatrix}
           \Omega(m)& -1\\
           1& 0\\
         \end{pmatrix}
$$
We consider the orthogonal basis of $E^c_m$\ given by
$$
s_m=(1,\Omega(m));\quad u_m=(\Omega(m),-1)
$$
and verify that
\begin{equation}
d_mf(s_m)=(0,1),\quad d_mf(u_m)=(1+\Omega^2(m),\Omega(m)).
\end{equation}
Now if $X$\ is a unit vector field  tangent to $E^c$\ we can write
\begin{equation}
X_m=\frac{\cos{\theta^X(m)}}{\sqrt{1+\Omega^2(m)}}s_m+\frac{\sin{\theta^X(m)}}{\sqrt{1+\Omega^2(m)}}u_m,
\end{equation}
where $\theta^X(m)$ is the angle $\angle(X_m,s_m)$. Hence
\begin{equation}
d_mf(X_m)=\left(\sin \theta^X(m)\sqrt{1+\Omega^2(m)},\frac{\cos \theta^X(m)+\sin \theta^X(m) \Omega(m)}{\sqrt{1+\Omega^2(m)}}\right)
\end{equation}
which in turn implies
\begin{equation}\label{normdif}
\begin{split}
\norm{d_mf(X_m)} &\geq |\sin{\theta^X(m)}|\sqrt{1+\Omega^2(m)}\\
&\geq |\sin{\theta^X(m)}||2+N\cos{ x}|\\
&\geq N|\sin{\theta^X(m)}||\cos{ x}|-2.
\end{split}
\end{equation}
\begin{definition}
 For $0\leq a<b\leq 1$, \emph{a strip} $S[a,b]$\ is a region of the form
\begin{equation*}
S[a,b]=\{ m=(x,y,z,w)\in M:x\in[a,b]\}.
\end{equation*}
The \emph{length} of the strip $S[a,b]$\ is $l(S[a,b])=b-a$.

By a harmless abuse of language we also call strips to the union of two strips, and extend the concept of length accordingly.

The \emph{critical strip} is the strip
$$
\mathit{Crit}=S[b_1,b_2]\cup S[b_3,b_4]
$$
where $0<b_1<\pi/2<b_2<b_3<3\pi/2<b_4<2\pi$\ are such that
\begin{equation*}
\cos{ b_1}=\cos{ b_4}=\frac{1}{\sqrt{N}}\qquad \cos{ b_2}=\cos{ b_3}=-\frac{1}{\sqrt{N}}.
\end{equation*}
\end{definition}

One verifies that
\begin{equation}\label{positionbs}
\begin{split}
&-\frac{2}{\sqrt{N}}<b_1-\frac\pi2<-\frac{1}{\sqrt{N}},\quad \frac{1}{\sqrt{N}}<b_2-\frac\pi2<\frac{2}{\sqrt{N}}\\
&-\frac{2}{\sqrt{N}}<b_3-\frac{3\pi}2<-\frac{1}{\sqrt{N}},\quad \frac{1}{\sqrt{N}}<b_4-\frac{3\pi}2<\frac{2}{\sqrt{N}}.
\end{split}
\end{equation}

In particular, $l(\mathit{Crit})\leq \frac{8}{\sqrt{N}}$.
\begin{lemma}\label{Xnobad}
Let $X$ be a vector field tangent to the center bundle. For $m$\ outside the critical strip we have
$$
|\Omega(m)|>\sqrt N-2,\quad \mathrm{and}\quad
\norm{df(X_m)}\geq \sqrt{N}|\sin{\theta^X(m)}|-2.
$$
\end{lemma}
\begin{proof}
If $m\in \mathit{Crit}^c$\ then
$|N\cos{x}|\geq \sqrt{N}$ and the claim follows from equation \eqref{normdif}.
\end{proof}
Let us define the following cone:
$$\Delta:=\mathbb R\cdot \{(1,n):\; |n|\le \sqrt[4]  N\}\subset\Real^2.$$
\begin{definition}
An adapted field $(\gamma, X)$\ is called \emph{good} if
$$
\forall m\in \gamma,\quad X_m\in \Delta,
$$
otherwise it is called \emph{bad}.
\end{definition}

This Manicheistic dichotomy is suitable to evaluate the  expectation of vector growth.

\begin{proposition}\label{Bounds}
For $N$\ sufficiently large, for every  adapted field $(\ga,X)$, it holds:
\begin{enumerate}
\item $E(\ga,X)\geq -\log{2N}.$
\item  Furthermore, if $(\ga,X)$\ is good then $E(\ga,X)\geq \frac{1}{7}\log{N}.$
\end{enumerate}
\end{proposition}
The following Lemma will be useful to prove the above proposition:
\begin{lemma}\label{Xgood}
For $N$ sufficiently large, for every  good adapted field $(\gamma, X)$, it holds:
\[|\sin{\theta^X(m)}|\geq \frac{1}{\sqrt[3]{N}}\quad \forall m\notin \mathit{Crit}.\]
\end{lemma}
\begin{proof}
We compute for $m$ off the critical strip:
\[|\sin{\theta^X(m)}|\ge |\sin\angle(X_m,(0,1))|- |\angle (s_m,(0,1))|\ge  \frac{1}{2\sqrt[4]{N}} -\arcsin\frac{1}{\sqrt{1+\Omega^2(m)}}\]
and conclude the claim by Lemma \ref{Xnobad}.
\end{proof}
\begin{proof}[Proposition \ref{Bounds}]
The first claim follows directly from the fact that $\norm{df^{-1}|E^c}\leq 2N$. Assume that $(\ga,X)$\ is good, then
\begin{equation}
|\ga| \cdot E(\ga,X) = \int_{\ga\setminus\mathit{Crit}^c} \log \norm{d_mf(X)}d\ga + \int_{\ga\cap\mathit{Crit}} \log \norm{d_mf(X)}d\ga.
\end{equation}
By Lemmas \ref{Xnobad} and \ref{Xgood}, outside of the critical strip
\begin{equation}\label{fgrand}
\norm{d_mf(X_m)}\geq \sqrt[6]{N}-2
\end{equation}
and hence for sufficiently large $N$, with $|\epsilon|\le \lambda^{2N}$\
$$|\ga| \cdot E(\ga,X)\geq (1-\frac{8+\epsilon}{\sqrt{N}})(\frac{1}{6}\log(N)-2)|\gamma|-\frac{8+\epsilon}{\sqrt{N}}\log{2N} |\ga|
\geq \frac{\log{N}|\gamma|}{7}.$$
\end{proof}
As we said, we continue working with $N$\ for which the previous results hold.

\section{Transitions.}\label{Sect.5}

Now that we have concrete bounds for $E(\ga,X)$,\ we want to understand the proportion of good fields obtained after iterating a given one.
Recall that $f^k\circ\ga=\ga_1^k\ast\cdots\ga^k_{[\mu^{2kN}]}\ast \ga^k_{[\mu^{2kN}]+1}$. We define for every adapted field $(\gamma, X)$:
\begin{equation*}
\begin{split}
&G_k=G_k(\ga,X)=\left\{1\leq j\leq [\mu^{2kN}]:\Big(\ga_j^k,\frac{f^k_{\ast} X}{\norm{f^k_{\ast} X}}\Big)\text{ is Good}\right\}\\
&B_k=B_k(\ga,X)=\left\{1\leq j\leq [\mu^{2kN}]:\Big(\ga_j^k,\frac{f^k_{\ast} X}{\norm{f^k_{\ast} X}}\Big)\text{ is Bad}\right\}
\end{split}
\end{equation*}

\begin{lemma}\label{goodtogood}
If $(\ga,X)$\ is a good adapted field and $f^{-1}\ga^1_j\cap \mathit{Crit}\neq \emptyset$, then the field $(\ga^1_j,\frac{f_{\ast}X}{\norm{f_{\ast}X}})$\ is good.
\end{lemma}
\begin{proof}
Let $m\notin \mathit{Crit}$. For every $|n|\le \sqrt[4] N$, we recall that
\[d_mf(1,n)= (\Omega(m)-n,1).\]
As $|\Omega(m)-n| \ge |\Omega(m)|-|n|\ge \sqrt N- 2 -\sqrt[4] N$ by Lemma \ref{Xnobad}, it comes $d_mf(1,n)\in \Delta$ and the lemma follows.\end{proof}

\begin{lemma}\label{badtogood}
For every $N$ sufficiently large, for every bad adapted field $(\ga,X)$ there exists a strip $S_X$ of length  $\pi$ so that
for every $j$ satisfying  $f^{-1}\ga^1_j\subset  S_X$, the field $(\ga^1_j,\frac{f_*{\ast}X}{\norm{f_*{\ast}X}})$\ is good.
\end{lemma}
\begin{proof}
For $m=(x,y,z,w)$ and $m_0=(x_0,y_0,z_0,w_0)\in M$,\ one has
\begin{equation}\label{s0}
df_m(X_m)=d_mf(X_m-X_{m_0})+d_mf(X_{m_0})
\end{equation}
where by Corollary \ref{variation} the first vector on the right hand side has norm less than $2N\lam^{N/3}$. 





By the ``badness'' hypothesis there exists $m_0$ such that $X_{m_0}$ is collinear with a vector $(1,n)$ with $|n|\geq \sqrt[4]{N}$.
 For every $m\in\ga$, the vector $df_m(X_{m_0})$ is collinear with the vector $(N\cos x +2-n,1)$. Hence, for every $m$ such that $\cos x$ is of the same sign as $-n$, the vector  $df_m(X_{m_0})$ is large and makes a small angle with the $x$ axis, and thus is in $\Delta$, which implies that $df_m(X_{m})\in\Delta$ as well. Such condition on the cosine corresponds to a strip of length $\pi$.
\end{proof}

These lemmas enable us to bound the ratio of the number of transitions between  bad and good adapted fields.
Indeed, for every $u$-curve $\gamma$, the elements of the partition
$(f^{-1}(\gamma_j^1))_{1\leq j\le [\mu^{2N}]}$ have very small and  comparable length when $N$ is large. Since the first coordinate of $\gamma$ has almost constant derivative we can deduce the following Proposition from the two above lemmas:

\begin{proposition}\label{G1B1} For every bad adapted field $(\gamma,X)$, we have:
$$\#G_1(\gamma, X)\ge \frac{\mu^{2N}}{3},\quad \#B_1(\gamma, X)\leq \frac{2\mu^{2N}}{3}\quad \text{with $\#$ the cardinality of a set}. $$
For every good adapted field $(\gamma,X)$, we have:
$$\#G_1(\gamma, X)\ge (1-\frac{10}{2\pi\sqrt{N}})\mu^{2N},\quad \#B_1(\gamma, X)\leq \frac{10}{2\pi\sqrt{N}}\mu^{2N}. $$
\end{proposition}

\begin{proof}
By the previous Proposition, if $(\gamma,X)$ is bad, there exists a band $S_X$ of length equal to $\pi$ such that for every $j$ satisfying  $f^{-1}\ga^1_j\subset  S_X$ the adapted field $(\ga^1_j,\frac{(f_N){\ast}X}{\norm{(f_N){\ast}X}})$\ is good. This corresponds to almost half of curves $\ga^1_j$. The second part is analogous, using $l(\mathit{Crit})\leq \frac{8}{\sqrt N}$.
\end{proof}

Let $\eta=\eta_N:=\frac{5}{\pi\sqrt{N}}$. The Proposition permits us to readily calculate

\begin{equation}\label{GyB}
\begin{split}
&\#G_{k+1}\geq (1-\eta)\mu^{2N}\#G_k+\frac{1}{3}\mu^{2N}\#B_k\\
&\#B_{k+1} \leq \eta\mu^{2N}\#G_k+\frac{2}{3}\mu^{2N}\#B_k+ \mu^{2N}.
\end{split}
\end{equation}
The last term $+ \mu^{2N}$ is given by the possible bad curves coming from the slice $\gamma_{[\mu^{2N}]+1}^{k}$ of $f^{k}\gamma$  which is not a $u$-curve.
\begin{lemma}\label{doce}
If $N$\ is sufficiently big then for all good adapted fields $(\ga,X)$ and for all  $k\geq 1$\ we have $\#G_{k}>100 \#B_{k}$.
\end{lemma}
\begin{proof}
The proof is done by induction. By hypothesis $\#G_{0}>100 \#B_{0}=0$. Obverse also that $\#G_{1}>100 \#B_{1}\not=0$. Assume the claim was established for $k\ge 1$. Then by the previous Lemma and the induction hypothesis
\begin{equation*}
\begin{split}
\frac{\#B_{k+1}}{\#G_{k+1}}\le \frac{\eta\#G_k+\frac{2}{3}\#B_k+1}{(1-\eta)\#G_k+\frac{1}{3}\#B_k} \le
\frac{\eta\#G_k+\frac{2}{300}\#G_k+1}{(1-\eta)\#G_k}
\le
\frac{\eta+\frac{2}{300}}{1-\eta} +\frac{1}{2 \mu^{2N}}
\end{split}
\end{equation*}
The limit of the last term when $N\mapsto \infty$  is smaller  than $1/150<1/100$, hence for $N$ sufficiently large (independently of $k$), $\#G_{k+1}> 100\#B_{k+1}$.
\end{proof}

\begin{proof}[Proof of Theorem \ref{main} for $f_N$]\

Take $(\ga,X)$\ a good adapted field and recall that (\ref{poursct6}):
\begin{multline}\label{intepourf}
I_n^{\ga,X}=\sum_{k=1}^{n-1}\frac{1+\epsilon_k}{\mu^{2Nk}|\ga|}\left(\sum_{j=0}^{[\mu^{2kN}]}\int_{\ga_j^k} \log{\norm{d_{m}f(Y^k)}}d\gamma_{j}^k\right.\\
\left.+\int_{\ga^k_{[\mu^{2kN}]+1}} \log{\norm{d_{m}f(Y^k)}}d\ga^k_{[\mu^{2kN}]+1}\right)
\end{multline}
Observe that the length of $u$-curves is almost constant by Corollary \ref{partialcoro}, and so $|\ga_j^k|>\frac{9}{10}|\ga|$. Note that by the previous Lemma $\#G_k>\frac{1}{1+\frac1{100}}[\mu^{2kN}]$. Splitting between bad and good adapted fields and using Proposition \ref{Bounds}, we deduce\footnote{Observe that for $k=0, B_0=0$.}
\begin{equation}
\begin{split}
&\frac{I_n^{\ga}}{n}>\frac{1}{n}\sum_{k=0}^{n-1} \frac{9(1+\epsilon_k)}{10\mu^{2Nk}} \left(\#G_k\frac{\log{N}}{7}-(\#B_k+1)\log{2N}\right)\\
& >\frac{1}{n}\sum_{k=0}^{n-1}\frac{9(1+\epsilon_k)}{10}\left(\frac{1}{1+\frac1{100}}\big(\frac{\log{N}}{7}-\frac{\log{2N}}{100}\big)-\frac{\log{2N}}{\mu^{2kN}}\right)>\frac{1}{20}\log{N}.
\end{split}
\end{equation}
An application of Proposition \ref{ergo} finishes the proof.
\end{proof}

\section{Robustness of Non-Uniform Hyperbolicity}\label{Sect.6}

In this section we indicate the relevant changes to our previous procedure to prove Theorem \ref{main}, namely that there is a small neighborhood $U\subset Diff^2_{leb}(M)$\ of $f_N$\ formed by $C^2$-conservative diffeomorphisms having a NUH physical measure. We take $N$ large such that the main Theorem and its intermediate results hold for $f_N$.
The neighborhood $U$ will be chosen small depending on $N$ (which is supposed large).

We start by fixing $U$\ such that every $g\in U$\ is partially hyperbolic, and
\begin{enumerate}
\item[(\textit{A})] for all unit vectors $v^s\in E^s_g$,
$v^c\in E^c_g$ and $v^u\in E^u_g$, the following inequalities holds:
$$0.99\lam^{2N}\leq \norm{dg(v^s)} \leq 1.01 \lam^{2N}<\!\!<1$$
$$1<\!\!<0.99\mu^{2N}\leq \norm{dg(v^u)} \leq 1.01 \mu^{2N}$$
$$\frac{1}{2N}\leq \norm{dg(v^c)} \leq 2N.$$
\item[(\textit{B})] the following bounds hold $\norm{d^2g}, \norm{d^2g^{-1}}\leq 2N.$
\item[(\textit{C})] $E^c_g$\ is $1/2$-H\"older.
\end{enumerate}

We comment on the last point. Given a partially hyperbolic map $g$ consider constants  $\varsigma,\upsilon,\varrho,\widehat{\varrho},\widehat{\upsilon},\widehat{\varsigma}$ such that for all unit vectors $v^s\in E^s_g$,
$v^c\in E^c_g$ and $v^u\in E^u_g$
:
\[
\varsigma<\norm{d_mg(v^s)}<\upsilon,\quad
\varrho<\norm{d_mg(v^c)}<\widehat{\varrho}^{-1},\quad
\widehat{\upsilon}^{-1}<\norm{d_mg(v^u)}<\widehat{\varsigma}^{-1}
\]
\begin{theorem}
If $g$ is of class $\mathcal{C}^2$ and $\theta\in (0,1)$ satisfies
$$
\upsilon<\varrho\varsigma^{\theta}, \widehat{\upsilon}<\widehat{\varrho}\widehat{\varsigma}^{\theta}
$$
then $E^c_g$ is $\theta$-H\"older.
\end{theorem}

See Section 4 in \cite{HolFolRev} for a discussion. Applying this theorem and using (\textit{A}), we conclude that for $N$ large enough and then $U$ small enough, for every $g\in U$ the bundle $E^c_g$ is $1/2$-H\"older. Moreover, regarding  $E^c_g$ as a section of the corresponding Grassmanian, its H\"older constant is small when $g$ is close to $f$, since the corresponding to $E^c_f$  is equal to zero.


\begin{definition}
A $u$-\emph{curve} for $g$\ is a curve $\ga=(\ga_x,\ga_y,\ga_z,\ga_w)\colon [0,2\pi]\rightarrow M$\ tangent to $E^u_g$\ and such that $|\frac{d\ga_x}{dt}(t)|= 1,\ \forall t\in [0,2\pi] $. For every $k\geq0$ there exists an integer  $N_k=N_k(\ga)$ such that the curve $g^k\circ\ga$\ can be written as a concatenation
\begin{equation}\label{concaug} g^k\circ\ga=\ga_1^k\ast\cdots\ga_{N_k}^k\ast \ga_{N_k+1}^k\end{equation}
where $\ga_j^k\ j=1,\ldots, N_k$\ are $u$-curves and $\ga^k_{N_k+1}$\ is a segment of $u$-curve.
\end{definition}

The unstable bundle depends continuously on the perturbation, and thus we obtain the following analogue to Corollary \ref{partialcoro}, for $U$ small enough.

\begin{corollary}\label{partialcorog}
Let $m\in M$\ and let $v=(v_x,v_y,v_z,v_w)$ be a unit vector in $E^{u}_g(m)$. Then
$$
\lam^{N}(\norm{P_x(e^u)}-3\lam^N)\leq {|v_x|}\leq \lam^{N}(\norm{P_x(e^u)}+3\lam^N).
$$
In particular the length of a $u$-curve is in $[(\lam^{N}(\norm{P_x(e^u)}+3\lam^N))^{-1}, (\lam^{N}(\norm{P_x(e^u)}-3\lam^N))^{-1}]$ and  every pair of $u$-curves $(\gamma,\gamma')$ for $g$ satisfy:
\[0.9 \cdot  Leb(\gamma)\le  Leb(\gamma')\le 1.1 \cdot  Leb(\gamma).\]
\end{corollary}

\begin{definition}
An adapted field $(\ga,X)$ for $g$\ is the pair of  a $u$-curve $\ga$ and of a unit vector field  $X$ satisfying:
\begin{enumerate}
\item $X$\ is tangent to the center direction $E^c_g$.
\item $X$ is $(C_X,1/2)$-H\"older along $\ga$, with $C_X<20 N^2\lam^{N}$.
\end{enumerate}
\end{definition}


The action of the map $f$ on unstable leaves corresponds almost to the mere multiplication by a constant factor. To deal with the perturbative case we need distortion bounds. Given a map $g\in U$ and an integer $k$ we denote by $J^u_{g^k}$ the \emph{unstable Jacobian} of $g^k$,  namely
$$
J^u_{g^k}(m):=|det(d_mg^k|_{E^u_g})|, \quad \forall m\in M.
$$
Note that from (\textit{A}) follows
\begin{equation}\label{eqjacobiano}
\forall m\in M,\quad \lam^{2N}/1.01\leq J^u_{g^{-1}}(m) \leq \lam^{2N}/0.99
\end{equation}

\begin{lemma}[Distortion bounds]\label{distortion}
There exists a constant $D=D_N$ with the following property. For all $g\in U$ and  $u$-curve $\ga$ for $g$, for every $k\geq 0$, it holds
$$
\forall m,m'\in \ga,\quad  \frac{1}{D}\leq \frac{J^u_{g^{-k}}(m)}{J^u_{g^{-k}}(m')} \leq D.
$$
Furthermore, for $N$ large and for $U$ small, the number $D$ can be taken close to $1$.
\end{lemma}

This Lemma is classical. We present the proof nonetheless, since the version given is adapted to our purposes.

\begin{proof}
Fix $g\in U$ and $u$-curve $\ga$  for $g$. For $m,m'\in \ga$ and an integer $j$, we denote by $m_j=g^j(m)$, $m'_j=g^j(m')$. We compute
\begin{equation*}
\begin{split}
\log\frac{J^u_{g^{-k}}(m)}{J^u_{g^{-k}}(m')}&=\sum_{j=-k+1}^0 \log\left|\frac{det(d_{m_j}g^{-1}|_{E^u_g})}{det(d_{m'_j}g^{-1}|_{E^u_g})}\right|\leq \sum_{j=-k+1}^0 \left|\frac{det(d_{m_j}g^{-1}|_{E^u_g})-det(d_{m'_j}g^{-1}|_{E^u_g})}{det(d_{m'_j}g^{-1}|_{E^u_g})}\right|\\
&\leq \frac{1.01}{\lam^{2N}}\sum_{j=-k+1}^0 \left|det(d_{m_j}g^{-1}|_{E^u_g})-det(d_{m'_j}g^{-1}|_{E^u_g})\right| \quad (\text{by equation \eqref{eqjacobiano}})\\
&\leq \frac{1.01}{\lam^{2N}} \sum_{j=-k+1}^0 C_0d(m_j,m'_j) \leq \frac{1.01C_0}{\lam^{2N}}\sum_{j=-k+1}^0 \Big(\frac{\lam^{2N}}{0.99}\Big)^{|j|}d(m_0,m'_0)
\end{split}
\end{equation*}
where $C_0$ is the Lipschitz constant of $\det dg|_{E^u}$ restricted to $\gamma$. Note that $C_0$ is small when $N$ is large and $g$ is $\mathcal{C}^2$-close to $f$. Altogether we conclude
$$
\log\frac{J^u_{g^{-k}}(m)}{J^u_{g^{-k}}(m')}\leq \frac{1.01C_0}{\lam^{2N}}\frac1{1-\frac{\lam^{2N}}{0.99}}d(m_0,m'_0)
$$
and since the length of $u$-curves is bounded from above the claim follows.
\end{proof}

From this Lemma and using the change of variables Theorem we obtain:

\begin{corollary}\label{corodistbound}
For all $g\in U$ and $u$-curve $\ga$  for $g$, for every measurable set $A\subset \ga$, for every $k\geq 0$ it holds
$$
\frac1{D}\frac{Leb(A)}{Leb(\ga)}\leq \frac{Leb(g^{-k}A)}{Leb(g^{-k}\ga)}\leq D\frac{Leb(A)}{Leb(\ga)}.
$$
\end{corollary}

The proof of the existence of the NUH physical measure for the maps $g\in U$ is obtained by analyzing the quantity

\begin{equation*}
I_n^{\ga,X}=\frac{1}{|\ga|}\int_{\ga} \log\norm{d_{m}g^n(X)}d\ga
\end{equation*}
where  $(\ga,X)$\ is an adapted field for $g$.

\begin{proposition}\label{ergog}
Suppose that there exists $C>0$ with the following property: for every $u$-curve $\ga$, there exists a vector field $X$  such that $(\gamma, X)$ is an adapted field for $g$ and such that for all large $n\geq 0$\
$$
\frac{I_n^{\ga,X}}{n}>C.
$$
Then the map $g$\ has a positive Lyapunov exponent greater than $C/2$ in the center direction at Lebesgue almost every point.
\end{proposition}
\begin{proof}
We refer the reader to the proof of Proposition \ref{ergo} since the arguments are completely analogous. 
Consider the set $B$\ of regular points for which the map $g$\ does not have a positive Lyapunov exponent greater than $C/2$ in the center direction. By absolute continuity of the unstable foliation and Fubini's theorem, there exists an unstable manifold $\gamma$ which intersects $B$ at a  subset of positive $1$-Lebesgue measure. Let $b\in \gamma$ be a density point of such a subset.

Choose the curve $\ga^{\ep}:[-\ep,\ep]\rightarrow M$\ as
$$
\ga^{\ep}(t)=g^{-k}\circ \be_{k}(t)
$$
where $\be_{k}$\ is the $u$-curve for $g$\ satisfying $\be_{k}(0)=g^k(b)$ , and observe that the length $Leb(\ga^{\ep})$\ is small for $k$ large, although this curve is possibly not symmetric around $b$. If we consider the left and right segments of $\ga^{\ep}$ with common boundary $b$, we see that by the distortion bounds Lemma that their quotient is close to 1. With this and the absolute continuity of $\mathcal F^u$, as $b$ is density point of $B$, we can assume that for $\gamma^\epsilon$ small enough (or equivalently $k$ sufficiently large)
$$\frac{Leb(\ga^{\ep}\cap B^c)}{Leb(\ga^{\ep})}<\frac{C}{2D}.$$

Using again distortion bounds,  we can write for any point $m^k\in g^k\ga^{\ep}$\
 \begin{equation*}\label{distappcurves}J^u_{g^{-k}}(m^k)\ge \frac{Leb(\gamma^\epsilon)}{D Leb(g^k\gamma^\epsilon)}.\end{equation*}
Let $\chi(m)=\limsup_n \frac1n \log\|dg^n \circ X\circ g^k(m)\|$, for $m\in \gamma^\epsilon$ and with $X$ given by the Proposition hypothesis.
\begin{equation*}\begin{split}
\int_{\ga^{\ep}}\chi d\ga^{\ep}=&\int_{g^k\ga^{\ep}}\chi\circ g^{-k} J^u_{g^{-k}}d(g^k\ga^{\ep})\ge \frac{Leb(\gamma^\epsilon)}{D Leb(g^k\gamma^\epsilon)} \int_{g^k\ga^{\ep}} \chi\circ g^{-k} d(g^k\ga^{\ep})\\
&\ge \frac{Leb(\gamma^\epsilon)}{D Leb(g^k\gamma^\epsilon)} C\cdot Leb(g^k\ga^{\ep})\ge\frac{C}{D}Leb(\ga^{\ep}).
\end{split}
\end{equation*}
The rest of the proof is identical to Proposition \ref{ergo}.
\end{proof}

We proceed as we did in the case of $f$ and study the integral
$$
E(\ga,X):=\frac1{|\gamma|}\int_{\ga} \log \norm{d_mg(X)}d\ga
$$
where $(\ga,X)$ is an adapted field for $g$.

Denote by $\pi\colon \Real^4\rightarrow \Real^2\equiv E^c_f$\ the canonical projection, and for a vector field $X$ on a curve $\gamma$, let $$\widetilde{X_m}:=\frac{\pi(X_m)}{\norm{\pi(X_m)}}\quad\forall m\in \ga.$$

\begin{definition}
An adapted field $(\ga,X)$\ for $g$\ is good\footnote{We keep the definitions of $\mathit{Crit},\Delta$ given in Section 4.} if for every $m\in \ga,\ \widetilde{X}_m\in \Delta$.
\end{definition}

\begin{proposition}\label{Boundsg}
For $N$ large and $U$ small, for all $g\in U$ and $(\ga,X)$ adapted field  for $g$, it holds:
\begin{enumerate}
\item $E(\ga,X)\geq -\log{2N}.$
\item If  $(\ga,X)$\ is good then $E(\ga,X) \geq\frac{1}{7}\log{N}.$
\end{enumerate}
\end{proposition}

\begin{proof}
The first part follows directly from property (\textit{A}). For the second part, equation \eqref{fgrand} implies that for every $m\in \mathit{Crit}^c\cap\ga$ the norm $\norm{d_mf(\widetilde{X}_m)}\geq \sqrt[6]{N}-2$. Since the center bundle depends continuously with respect to the map, if $U$ is sufficiently small,
$X_m$ is close to $\widetilde{X}_m$ and $d_mg(X_m)$ is close to  $d_mf(\widetilde{X}_m)$.  
 Thus for every $m\in \mathit{Crit}^c$\
$$
\norm{d_mg(X_m)}\geq \norm{d_mf(\tilde X_m)}-1\geq \sqrt[6]{N}-3.
$$
The rest of the proof is completely analogous to Proposition \ref{Bounds}.
\end{proof}

Let $(\ga,X)$ be an adapted field for a map $g\in U$.  For $k\geq0$ consider the pairs $(\ga_j^k,Y^k|\ga_j^k)$, $1\leq j\leq N_k$\ with $Y^k=\frac{g^k_{\ast}X}{\norm{g^k_{\ast}X}}$.

In the same fashion as in Lemma \ref{adapted}, we have:
\begin{lemma}
Every possible pair $(\ga_j^k,Y^k|\ga_j^k)$\ is an adapted field.
\end{lemma}

Observe that
\[I_n^{\ga,X}=\sum_{k=0}^{n-1}\frac{1}{|\ga|}\int_{\ga} \log\norm{d_{g^km}g(Y^k\circ g^k(m))}d\ga= \sum_{k=0}^{n-1}\frac{1}{|\ga|}\int_{g^k\ga} \log\norm{d_{m}gY^k}J^u_{g^{-k}}d(f^k(\ga))\]
Writing $g^k \gamma$ as the concatenation (\ref{concaug}), it comes as for (\ref{intepourf}):
\[I_n^{\ga,X}= \sum_{k=0}^{n-1}\left(R_k+ 
\sum_{j=0}^{N_k}
\frac{1}{|\ga|} \int_{\ga_j^k} \log{\norm{d_{m}g(Y^k)}}J^u_{g^{-k}}d\ga_j^k
\right),
\]
where $R_k= \frac{1}{|\ga|} \int_{\ga^k_{N_k+1}} \log{\norm{d_{m}g(Y^k)}}J^u_{g^{-k}} d\ga_{N_k+1}^k$. By Corollary \ref{partialcorog}, for every $j, k$:
\[\frac{|\ga_j^k|}{|\ga|}\ge 0.9 \quad\text{and}\quad \frac{|\ga_{N_k+1}^k|}{|\ga|}\le 1.1\]
Hence 
\[I_n^{\ga,X}\ge  \sum_{k=0}^{n-1}\left(R_k+ 
0.9 \sum_{j=0}^{N_k}
\min_{\ga_j^k} (J^u_{g^{-k}}) \cdot E(\ga_j^k, Y^k )
\right). 
\]
Also by {\it(A)} and Corollary \ref{partialcorog}:
\[\quad |R_k|\le |\ga_{N_k+1}^k| \max_{\ga_{N_k+1}^k} | J^u_{g^{-k}}|\cdot  \log{\norm{d g|E_c}}\le \frac{(0.99\mu^{2N})^{-k}\log( 2N)}{\lam^{N}(1-2\lam^N)\norm{P_x(e^u)}},\]
Which approaches $0$ when $k$ approaches infinity. Hence its Ces\`aro mean also converges to zero:
\[\frac1n \sum_{k=0}^{n-1} |R_k|\rightarrow 0,\; n\rightarrow \infty.\] 
By using again the Ces\`aro mean and  Proposition \ref{ergog}, to show that the map $g$\ has a positive Lyapunov exponent in the center direction, it suffices to show that:
\begin{proposition} \label{fonda} For $U$ small enough, for every $g\in U$, every admissible field $(\gamma,X)$ for $g$, every $k\ge 0$, it holds
\[ \sum_{j=0}^{N_k}
\min_{\ga_j^k}(J^u_{g^{-k}} \cdot E(\ga_j^k, Y^k) )\ge \frac{\log N}{1000}.\]
\end{proposition}
We recall that by the symmetry property of $f_N$, the inverse of $g$ has the same property as $g$ and so the same argument for $g^{-1}$ implies the main result. 
\begin{proof}[Proof of Proposition \ref{fonda}]
This follows from a lemma based on the same dichotomy  between Good and bad adapted fields for $g$. Let $G_k(\gamma, X) $ and $B_k(\gamma,X)$ be the corresponding subsets of $\{\gamma_j^k,j\le N_k\}$. Let us begin with the following lemma proved below.

\begin{lemma}\label{prefonda} For $N$ large and then $U$ small enough, for every $g\in U$, for every good adapted field $(\ga,X)$ for $g$, for every $k\geq 0$ it holds
\begin{enumerate}
\item $\sum_{j\in G_k} \min_{\ga_j^k}J^u_{g^{-k}}\geq \frac{1}{100}.$
\item $\sum_{j\in G_k} \min_{\ga_j^k}J^u_{g^{-k}}\geq 100 \sum_{j\in B_k} \max_{\ga_j^k}J^u_{g^{-k}}.$
\end{enumerate}
\end{lemma}
Indeed:
\[ \sum_{j=0}^{N_k}
\min_{\ga_j^k} (J^u_{g^{-k}}\cdot E(\ga_j^k, Y^k)) =
\sum_{j\in G_k} \min_{\ga_j^k}(J^u_{g^{-k}}\cdot E(\ga_j^k, Y^k))+\sum_{j\in B_k} \min_{\ga_j^k}(J^u_{g^{-k}}\cdot E(\ga_j^k, Y^k))\]
By Proposition \ref{Boundsg} and then Lemma \ref{prefonda}:
\[ \sum_{j=0}^{N_k}
\min_{\ga_j^k} (J^u_{g^{-k}}\cdot E(\ga_j^k, Y^k)) \ge \frac{log N}7\sum_{j\in G_k} \min_{\ga_j^k}J^u_{g^{-k}} -{log 2N}\sum_{j\in B_k} \max_{\ga_j^k}J^u_{g^{-k}}\]
\[\ge (\frac{log N}7-\frac{ log 2N}{100})\sum_{j\in G_k} \min_{\ga_j^k}J^u_{g^{-k}}\ge \frac{\frac{log N}7-\frac{log 2N}{100}}{100}\ge \frac{\log N}{1000}\] 
which is the claim of Proposition \ref{fonda}.\end{proof}

\begin{proof}[Proof of Lemma \ref{prefonda}] We start with the following
\begin{lemma}\label{numberg}
For every adapted field $(\ga,X)$  for $g$. For every positive integer $k\geq0$, it holds
$$
\frac1{2D}\leq\sum_{j\in G_k} \min_{\ga_j^k}J^u_{g^{-k}}+ \sum_{j\in B_k} \max_{\ga_j^k}J^u_{g^{-k}}\le 2D.
$$
\end{lemma}
\begin{proof}
Form the distortion bounds Lemma \ref{distortion} and Corollary \ref{partialcorog}, we get
\begin{equation*}
1=\frac{1}{|\ga|}\int_{\ga}d\ga=\frac{1}{|\ga|}\sum_{k=1}^{N_k+1}\int_{\ga_j^k}  J^u_{g^{-j}}(m^{j,k})d\ga_j^k
\geq \sum_{j\in G_k} \frac{|\ga_j^k|}{|\gamma|} \min_{\ga_j^k}J^u_{g^{-k}}+ \frac{1}{D}\sum_{j\in B_k\cup\{N_k+1\}} \frac{|\ga_j^k|}{|\gamma|}\max_{\ga_j^k}J^u_{g^{-k}}
\end{equation*}
\begin{equation*}
\Rightarrow 1\ge \frac{0.9}{D}
\left(\Big(\sum_{j\in G_k}  \min_{\ga_j^k}J^u_{g^{-k}}\Big)+ \Big(\sum_{j\in B_k} \max_{\ga_j^k}J^u_{g^{-k}}\Big) -\frac{1.1(0.99\mu)^{-2kN}}{\lam^{N}(1-2\lam^N)\norm{P_x(e^u)}}\right)
\end{equation*}
and the second inequality follows for $N$ large. The first one is similar.
\end{proof}
As $D$ is close to $1$ when $U$ is small, the latter lemma with the second statement of Lemma \ref{prefonda} imply the first statement of Lemma  \ref{prefonda}. Hence it remains only to show the second statement of Lemma \ref{prefonda}.
For this end, we study the transitions between good and bad adapted fields.
\begin{lemma}\label{tripforg}
For $U$ small, for every $g\in U$, every   adapted field $(\ga,X)$
\begin{enumerate}
\item 
If $(\ga,X)$\ is a good adapted field and if $j$ is so that $g^{-1}\ga^1_j$ does not intersect the strip $\mathit{Crit}$ of length
$4/\sqrt{N}$, then the field $(\ga^1_j,\frac{g_{\ast}X}{\norm{g_{\ast}X}})$\ is good.
\item 
If $(\ga,X)$\ is bad,  there exists a strip $S_X$ of length $\pi$ so that
for every $j$ satisfying  $g^{-1}\ga^1_j\subset  S_X$, the field $(\ga^j_1,\frac{g_{\ast}X}{\norm{g_{\ast}X}})$ is good.
\end{enumerate}
\end{lemma}

\begin{proof}
First observe that $\widetilde{g_{\ast}X}$ is close to $\frac{g_{\ast}\widetilde{X}}{\norm{g_{\ast}\widetilde{X} }}$ which is itself close to $\frac{f_{\ast}\widetilde{X}}{\norm{f_{\ast}\widetilde{X} }}$, for $U$ small and this is uniform for $g\in U$ and $(\gamma, X)$ adapted field for $g$. 

Assume that $(\ga_g,X)$ is good. In Lemma \ref{goodtogood} we proved that for every $m\in\ga_g\cap\mathit{Crit}^c$\
$$\frac{(f_{\ast}\widetilde{X})_{fm}}{\norm{(f_{\ast}\widetilde{X})_{fm}}}\in \Delta$$
and is uniformly away from the boundary. Hence we conclude that for every point  $m\in \ga_g\cap\mathit{Crit}^c, (\widetilde{g_{\ast}X})_{gm}\in \Delta$\ and the first part follows. The second part is similar using Lemma \ref{badtogood}.

\end{proof}

We are now ready to prove the second statement of  Lemma \ref{prefonda}. Observe that this is the exact analog of Lemma \ref{doce} for the case of $f$.
%
%
The proof is by induction. The case $k=0$ follows by hypothesis. Assume the claim for $k\geq0$. By the distortion estimate:

\[
D \sum_{j\in G_{k+1}} Leb(\gamma_j^{k+1}) \min_{\gamma_j^{k+1}} J^u_{g^{-k-1}}\geq 
\int_{\sqcup_{j\in G_{k+1}} \gamma_j^{k+1}} J^u_{g^{-k-1}} d(g^{k+1}\gamma)\]
\[\ge 
\int_{\sqcup_{j\in G_{k+1}} \gamma_j^{k+1}\cap g(\sqcup_{i\in G_{k}} \gamma_i^{k}) } J^u_{g^{-k-1}} d(g^{k+1}\gamma)\ge \sum_{i\in  G_{k}} \int_{
\sqcup_{j\in G_{k+1}} \gamma_j^{k+1}\cap g( \gamma_i^{k})} J^u_{g^{-k-1}} d(g^{k+1}\gamma)
\]
\[\ge \sum_{i\in  G_{k}} \int_{
g^{-1}(\sqcup_{j\in G_{k+1}} \gamma_j^{k+1})\cap  \gamma_i^{k}} J^u_{g^{-k}} d(g^{k}\gamma)
\ge \sum_{i\in G_{k}} \min_{ \gamma_i^{k}} J^u_{g^{-k}}  Leb(g^{-1}(\sqcup_{j\in G_{k+1}} \gamma_j^{k+1})\cap  \gamma_i^{k})\]
By using Lemma \ref{tripforg}, with $\eta= \frac{5\sqrt{N}}{\pi}$, and then Corollary \ref{partialcorog}
\[Leb(g^{-1}(\sqcup_{j\in G_{k+1}} \gamma_j^{k+1})\cap  \gamma_i^{k})\ge (1-\eta) Leb( \gamma_i^{k})\ge 0.9(1-\eta) \max_j Leb( \gamma_j^{k+1})\]
Hence 
\begin{equation}\label{final} \sum_{j\in G_{k+1}} \min_{\gamma_j^k} J^u_{g^{-k-1}}\geq \frac{0.9(1-\eta)}D   \sum_{i\in \in G_{k}} \min_{ \gamma_i^{k}} J^u_{g^{-k}} \end{equation}

Similarly:
\[
\frac1D \sum_{j\in B_{k+1}} Leb(\gamma_j^{k+1})\max_{\gamma_j^{k+1}} J^u_{g^{-k-1}}\leq 
\int_{\sqcup_{j\in B_{k+1}} \gamma_j^{k+1}} J^u_{g^{-k-1}} d(g^{k+1}\gamma)\]
\begin{align*}
&\leq \int_{\sqcup_{j\in B_{k+1}} \gamma_j^{k+1}\cap g(\sqcup_{i\in B_{k}} \gamma_i^{k}) } J^u_{g^{-k-1}} d(g^{k+1}\gamma)+
\int_{\sqcup_{j\in B_{k+1}} \gamma_j^{k+1}\cap g(\sqcup_{i\in G_{k}} \gamma_i^{k}) } J^u_{g^{-k-1}} d(g^{k+1}\gamma)\\
&+\int_{g(\gamma_{N_k+1}^k)} J^u_{g^{-k-1}} d(g^{k+1}\gamma).
\end{align*}
The last term can be bound as:
\[\int_{g(\gamma_{N_k+1}^k)} J^u_{g^{-k-1}} d(g^{k+1}\gamma) \le 1.01 \mu^{2N} Leb(\gamma_{N_k+1}^k) (1.01 \mu^{2N})^{-k-1}
\]\[\le 1.01\cdot Leb(\gamma_{N_k+1}^k) (1.01 \mu^{2N})^{-k}\le 1.01\cdot \lambda^{-N}(1-2\lambda^N)^{-1}\norm{P_x(e^u)}^{-1}(1.01 \mu^{2N})^{-k}\le \lambda^{N/2}.\]

Thus by Lemma \ref{tripforg},
\begin{equation*}
\begin{split}
\frac1D \sum_{j\in B_{k+1}} Leb(\gamma_j ^{k+1})\max_{\gamma_j^{k+1}} J^u_{g^{-k-1}}\leq &
 \sum_{i\in G_{k}} D\min_{ \gamma_i^{k}} J^u_{g^{-k}}  Leb(g^{-1}(\sqcup_{j\in B_{k+1}} \gamma_j^{k+1})\cap  \gamma_i^{k})\\
&+
\sum_{i\in B_{k}} \max_{ \gamma_i^{k}} J^u_{g^{-k}}  Leb(g^{-1}(\sqcup_{j\in B_{k+1}} \gamma_j^{k+1})\cap  \gamma_i^{k})+\lambda^{N/2}\\
\le& \sum_{i\in G_{k}} D\min_{ \gamma_i^{k}} J^u_{g^{-k}}  
\eta
Leb( \gamma_i^{k}) +
\sum_{i\in B_{k}} \max_{ \gamma_i^{k}} J^u_{g^{-k}}  \frac23 Leb(\gamma_i^{k})+\lambda^{N/2}.
\end{split} 
\end{equation*}

Therefore: 
\[ \sum_{j\in B_{k+1}} \max_{\gamma_j^k} J^u_{g^{-k-1}}\leq  1.1 D^2 \eta 
(\sum_{i\in G_{k}} \min_{ \gamma_i^{k}} J^u_{g^{-k}}  ) +
  \frac{2.2}3 D  (\sum_{i\in B_{k}} \max_{ \gamma_i^{k}} J^u_{g^{-k}})+\lambda^{N/2}D. \]
By the induction hypothesis:
\begin{equation}\label{final2} \sum_{j\in B_{k+1}} \max_{\gamma_j^k} J^u_{g^{-k-1}}\leq  (1.1 D^2 \eta 
+  \frac{2.2}{300} D)
(\sum_{i\in G_{k}} \min_{ \gamma_i^{k}} J^u_{g^{-k}}  )+\lambda^{N/2}D.
\end{equation}
Using equations (\ref{final}) and (\ref{final2}),  we finally conclude:
\[\frac{\sum_{j\in B_{k+1}} \max_{\gamma_j^k} J^u_{g^{-k-1}}}{\sum_{j\in G_{k+1}} \min_{\gamma_j^k} J^u_{g^{-k-1}}}\le \frac{(1.1 D^2 \eta 
+  \frac{2.2}{300} D)}{\frac{0.9(1-\eta)}D}+100\lambda^{N/2}D\]
which is less than $1/100$ for $\eta$, $\lambda^{N/2}$ and $(D-1)$ small enough, and this can be achieved by taking first $N$ large and then $U$ small.
\end{proof}
\section{Concluding Remarks.}

Exactly in the same fashion one can prove the robust non-uniform hyperbolicity of the map
$$
f_r(m)=(\mathbf{s}_r(x,y) + P_x\circ A^{[r]}(z,w), A^{[2r]}(z,w))
$$
for sufficiently large parameters $r\in \Real$.

Also, for every linear endomorphism $L$ of $\mathbb R^2$ such that $L_{[1,1]}\not=0$ the map
\[(x,y,z,w)\mapsto (\mathbf{s}_{N}(x,y)+L(z,w), A^N(z,w))\]
preserves the two dimensional Lebesgue measure of the center bundle. Its perturbations are close to satisfy the same. Hence, instead of using the reversibility of the standard map, one can use this almost area preserving property to deduce the existence of the contracting central Lyapunov exponent from the lower bound on the expending central Lyapunov exponents.

Moreover, it seems also possible to apply our method to the class of maps obtained when $\mathbf{s}_N$ is replaced
by \[(x,y)\mapsto (p x-y+N \sin(x), x),\] where $p\in \mathbb Z$ is arbitrary.
 

The presented integral method can also be used to give a new proof of the (ought to be) known robust non-uniform hyperbolicity of the map considered in \cite{RobTranT4}.

We can also wonder about the existence of a NUH physical measure for non conservative perturbations of $f_N$. However, very few works have been done on the existence of such measures when there are both positive and negative central Lyapunov exponents.

Let us point out that the map $f_N$\ can be $\mathcal{C}^2$-approximated by a stably ergodic diffeomorphism \cite{StableApprox}. We do not know if it $f_N$\  itself is (stably) ergodic. By \cite{ErgPH} stably ergodicity is implied by accessibility and center bunching (see the article for the definitions). In our case center bunching follows immediately, but accessibility seems harder to prove. Note that in the proof presented we did not need to use ergodicity, neither did we have to perturb the original dynamics to get it, as in other techniques.

%
%

We finish by noting the following. It is a result of J. Bochi and M. Viana \cite{BochVia} that for any closed symplectic manifold $(M,\omega)$ there exists a $\mathcal{C}^1$ generic set
$\mathcal{R}\subset Sym^1_{\omega}(M)$ such that for $g\in\mathcal{R}$ then either (a) at least two Lyapunov exponents of $g$ are zero Lebesgue almost everywhere,  or (b) $g$ is Anosov.

We observe below that our family of maps $(f_N)_N$ is symplectic. Also each $f_r$ cannot be $\mathcal{C}^0$-approximated by an hyperbolic map (see the remark after Proposition \ref{partialhyp}). We conclude that the aforementioned result of Bochi-Viana does not hold in the $\mathcal{C}^2$ setting.

\begin{corollary}
There exists an analytic symplectic map $f:\mathbb{T}^4\rightarrow \mathbb{T}^4$ and a neighborhood $U\subset Sym^2_{\omega}(\mathbb{T}^4)$ such that every $g\in U$ satisfies
\begin{enumerate}
 \item $g$ is not Anosov.
 \item All Lyapunov exponents for $g$ with respect to the Lebesgue measure are non-zero.
\end{enumerate}
\end{corollary}
\begin{proof}
Let us construct the sympletic form left invariant by $f_N$. Let $w_c$ be the alternate 2-form on $M$ such that restricted to the central bundle $w_c$ is the canonical volume form, and $Ker(w_c)=E^s\oplus E^u$. Let $d_u$ be the $1$-form equal to $1$ at $(\alpha_m, e^u)$ (see Proposition \ref{partialhyp} for the definition of $\alpha_m$), for every $m\in M$, and such that $Ker(d_u)= E^c\oplus E^s$. We notice that ${f_N}_*d_u= \mu^{2N} d_u$. Similarly, we define $d_s$ such that
${f_N}_*d_s= \lambda^{2N} d_s$ and $Ker(d_s)= E^c\oplus E^u$. Then the $2$-form $w= w_c+ d_u\wedge d_s$ is alternate and not degenerate, and so sympletic. It is also left invariant by $f_N$.
\end{proof}

\textbf{Acknowledgements:} We express our gratitude to the referee for his/her outstanding work, which greatly improved the readability and precision of the article. We also thank Omar A. Camarena for useful suggestions in the presentation.

\bibliographystyle{alpha}
\bibliography{biblio}

\end{document}